\documentclass[11pt]{scrartcl}

\usepackage[normalem]{ulem}

\usepackage{amsfonts,amsmath,amssymb,latexsym,soul,amsthm}
\usepackage[noadjust]{cite}
\usepackage{color,enumitem,graphicx}
\usepackage[colorlinks=true,urlcolor=blue,
=red,linkcolor=blue,linktocpage,pdfpagelabels,
bookmarksnumbered,bookmarksopen]{hyperref}
\usepackage[english]{babel}

\usepackage{authblk}

\usepackage[a4paper,left=1.0in,right=1.0in,top=3cm,bottom=3cm]{geometry}

\usepackage[T1]{fontenc}
\usepackage{lmodern}
\usepackage{microtype}

\usepackage{fixmath}

\numberwithin{equation}{section}

\newtheorem{Th}{Theorem}[section]
\newtheorem{Lem}[Th]{Lemma}
\newtheorem{Prop}[Th]{Proposition}
\newtheorem{Cor}[Th]{Corollary}
\theoremstyle{definition}
\newtheorem{Def}[Th]{Definition}
\theoremstyle{remark}
\newtheorem{Rem}[Th]{Remark}

\newcommand{\R}{\mathbb{R}}

\usepackage{todonotes}

\newcommand{\cF}{{\mathcal F}}

\newcommand{\cJ}{{\mathcal J}}

\newcommand{\cN}{{\mathcal N}}

\numberwithin{equation}{section}

\DeclareMathOperator*{\supp}{supp}

\DeclareSymbolFont{rsfs}{U}{rsfs}{m}{n}
\DeclareSymbolFontAlphabet{\mathscr}{rsfs}

\begin{document}


\title{Non-local to local transition for ground states of fractional Schr\"{o}dinger equations on bounded domains}

\author{Bartosz Bieganowski\thanks{Email address:
    \texttt{bartoszb@mat.umk.pl}}} \affil{\small Nicolaus Copernicus
  University \\ Faculty of Mathematics and Computer Science \\ ul. Chopina
  12/18, 87-100 Toru\'n, Poland}

\author{Simone Secchi\thanks{Email address:
    \texttt{Simone.Secchi@unimib.it}}} \affil{\small Dipartimento di
  Matematica e Applicazioni \\ Universit\`a degli Studi di
  Milano-Bicocca \\
  via Roberto Cozzi 55, I-20125, Milano, Italy}

\maketitle

\begin{abstract} 
  We show that ground state solutions to the nonlinear, fractional problem
  \begin{equation*} 
  \begin{cases}
    (-\Delta)^{s} u + V(x) u = f(x,u) &\quad \mathrm{in} \ \Omega, \\
    u = 0 &\quad \mathrm{in} \ \R^N \setminus \Omega,
\end{cases}
\end{equation*} 
on a bounded domain  $\Omega \subset \R^N$, converge (along a subsequence) in $L^2 (\Omega)$, under suitable conditions on $f$ and $V$, to a solution of the local problem as $s \to 1^-$.

\medskip

   
\end{abstract}

\tableofcontents

\section{Introduction}

The aim of this paper is to analyze the asymptotic behavior of least-energy solutions to the fractional Schr\"odinger problem
\begin{equation} \label{eq:1.1}
  \begin{cases}
    (-\Delta)^{s} u + V(x) u = f(x,u) &\quad \mathrm{in} \ \Omega, \\
    u = 0 &\quad \mathrm{in} \ \R^N \setminus \Omega,
\end{cases}
\end{equation}
in a bounded domain $\Omega \subset \mathbb{R}^N$. We recall that the fractional laplacian is defined as a singular integral via the formula
\begin{align*}
	(-\Delta)^s u (x) = C(N,s) \lim_{\varepsilon \to 0} \int_{\mathbb{R}^N \setminus B_\varepsilon(x)} \frac{u(x)-u(y)}{|x-y|^{N+2s}}\, dx\, dy
\end{align*}
with
\begin{align*}
\frac{1}{C(N,s)} = \int_{\mathbb{R}^N} \frac{1-\cos \zeta_1}{|\zeta|^{N+2s}}\, d\zeta_1 \cdots d\zeta_N.
\end{align*}
This formal definition needs of course a function space in which problem \eqref{eq:1.1} becomes meaningful: we will come to this issue in \ref{sect:variational-setting}.

\medskip

Several models have appeared in recent years that involve the use of the fractional laplacian. We only mention elasticity, turbulence, porous media flow, image processing, wave propagation in heterogeneous high contrast media, and stochastic models: see \cite{B,DPV,V,G}.

\medskip

Instead of \emph{fixing} the value of the parameter $s \in (0,1)$, we will start from the well-known identity (see \cite[Proposition 4.4]{DNPV})
\begin{align*}
\lim_{s \to 1^-} (-\Delta)^s u = -\Delta u, \quad\hbox{$u \in C_0^\infty(\mathbb{R}^N)$},
\end{align*}
and investigate the convergence properties of solutions to \eqref{eq:1.1} as $s \to 1$. In view of the previous limit, it is somehow natural to conjecture that solutions to \eqref{eq:1.1} converge to solutions of the problem
\begin{align} \label{eq:1.2}
\begin{cases}
-\Delta u + V(x) u = f(x,u) &\hbox{in $\Omega$} \\
u =0 &\hbox{in $\mathbb{R}^N \setminus \Omega$}.
\end{cases}
\end{align}
We do not know if this conjecture is indeed correct with this degree of generality, but we will prove that this happens --- up to a subsequence --- for \emph{least-energy} solutions. Our result extends the very recent analysis of Biccari \emph{et al.} (see~\cite{BHS}) in the \emph{linear} case for the Poisson problem to the semilinear case. See also \cite{BC}. 

\medskip

We collect our assumptions.
\begin{itemize}
\item[(N)] $N \geq 3$, $1/2<s<1$;
\item[($\Omega$)] $\Omega \subset \R^N$ is bounded domain with continuous boundary
$\partial \Omega$;
\item[(V)] $V \in L^\infty(\Omega)$ and $\inf_{\Omega} V >  0$;
\item[(F1)] $f \colon \Omega \times \mathbb{R} \to \mathbb{R}$ is a Carath\'eodory function, namely $f(\cdot, u)$ is measurable for any $u \in \R$ and $f(x, \cdot)$ is continuous for a.e. $x \in \Omega$. Moreover there numbers~are $C > 0$ and $p \in \left(2, \frac{2N}{N-1} \right)$ such that
\begin{align*}
|f(x,u)| \leq C (1 + |u|^{p-1})
\end{align*}
for $u \in \R$ and a.e. $x \in \Omega$.
\item[(F2)] $f(x,u)=o(u)$ as $u \to 0$, uniformly with respect to $x
  \in \Omega$.
\item[(F3)] $\lim_{|u| \to +\infty} \frac{F(x,u)}{u^2}=+\infty$
  uniformly with respect to $x \in \Omega$, where
  $F(x,u)=\int_0^u f(x,s)\, ds$.
\item[(F4)] The function $\R \setminus \{0\} \ni u \mapsto f(x,u)/u$ is strictly increasing on $(-\infty, 0)$ and on $(0, \infty)$, for a.e. $x \in \Omega$.
\end{itemize}

\begin{Rem} \label{rem:1.1}
  It follows from (F1) and (F2) that for
  every~$\varepsilon > 0$ there is~$C_\varepsilon > 0$ such that
\[
|f(x,u)| \leq \varepsilon |u| + C_\varepsilon |u|^{p-1}
\]
for every~$u \in \mathbb{R}$ and a.e~$x \in \Omega$.
Furthermore, assumption~(F4) implies the validity of the inequality
\begin{align*}
f(x,u)u \geq 2F(x,u)
\end{align*}
for every~$u \in \R$ and a.e.~$x \in \Omega$.
\end{Rem}

We can now state our main result.
\begin{Th}\label{th:main}
  Suppose that assumptions~(N), ($\Omega$), (V), (F1)--(F4) hold. For $1/2<s<1$, let $u_s \in H_0^s(\Omega)$ be a ground state
  solution of problem \eqref{eq:1.1}. Then, there is a sequence $\{s_n\}_n \subset (1/2, 1)$ such that $s_n \to 1$ as $n \to + \infty$ and $u_{s_n}$
  converges in $L^2(\Omega)$ to a ground state solution
  $u_0 \in H_0^1(\Omega)$ of the problem \eqref{eq:1.2}.
\end{Th}

\begin{Rem}
Actually, see Corollary \ref{cor:2.9}, it follows that  $u_{s_n}$ converges to $u_0$ in $L^\nu (\Omega)$ for every $2 \leq \nu < \frac{2N}{N-1}$.
\end{Rem}
\begin{Rem}
Unlike in \cite{BHS}, we cannot expect a convergence of the \emph{family} $\{u_s\}_{s}$ as $s \to 1^-$, since solutions to \eqref{eq:1.2} are not unique, in general.
\end{Rem}

\medskip

The paper is organized as follows. The second section contains a short introduction into fractional Sobolev spaces and the variational setting. In the third section we give the sketch of the proof of existence of ground states to \eqref{eq:1.1}. The fourth section is devoted to the proof of Theorem \ref{th:main}.

\section{The variational setting}\label{sect:variational-setting}

In this section we collect the basic tools from the theory of fractional Sobolev spaces we will need to prove our results. For a thorough discussion, we refer to \cite{MBRS,DNPV} and to the references therein.

\medskip

We define a Sobolev space on $\Omega$ as
\begin{align*}
  H^s(\Omega)= \left\{ u \in L^2(\Omega) \mid \int_{\Omega \times \Omega} \frac{|u(x)-u(y)|^2}{|x-y|^{N+2s}} \, dx\, dy < + \infty \right\},
\end{align*}
endowed with the norm
\begin{align*}
  \|u\|_{H^s(\Omega)}^2 = \|u\|_{L^2(\Omega)}^2 + \int_{\Omega \times \Omega} \frac{|u(x)-u(y)|^2}{|x-y|^{N+2s}} \, dx\, dy
\end{align*}
Furthermore, $H_0^s(\Omega)$ is the closure of $C_0^\infty(\Omega)$
with respect to the $H^s(\Omega)$-norm.
\begin{Def}
  For $0<s<1$, we define $X^s(\Omega)$ as the set of all measurable
  functions $u \colon \mathbb{R}^N \to \mathbb{R}$ such that the
  restriction of $u$ to $\Omega$ lies in $L^2(\Omega)$ and the map
\begin{align*}
  \mathbb{R}^N \times \mathbb{R}^N  \ni (x,y) \mapsto \frac{u(x)-u(y)}{|x-y|^{\frac{N}{2}-s}}
\end{align*}
belongs to $L^2(Q)$, where
\begin{align*}
  Q = \left( \mathbb{R}^N \times \mathbb{R}^N \right) \setminus \left( \Omega^c \times \Omega^c \right)
\end{align*}
and $\Omega^c = \mathbb{R}^N \setminus \Omega$.
 We also define
\begin{align*}
  X_0^s(\Omega) = \left\{ u \in X^s(\Omega) \mid \hbox{$u=0 $ a.e. in $\mathbb{R}^N \setminus \Omega$} \right\}.
\end{align*}
\end{Def}
It is well know, see \cite[Lemma 1.24]{MBRS}, that
$X^s(\Omega) \subset H^s(\Omega)$ with a continuous embedding, and
that
\begin{align*}
  X_0^s(\Omega) = \left\{ u \in H^s(\Omega) \mid \hbox{$u=0$ on $\Omega^c$} \right\}.
\end{align*}
Since we assume that~$\Omega$ has a continuous boundary
$\partial \Omega$, $C_0^\infty(\Omega)$ is dense in $X_0^s(\Omega)$
(see \cite[Theorem 2.6]{MBRS}), so that actually
$X_0^s(\Omega)=H_0^s(\Omega)$ for such a domain $\Omega$.  For
$u \in X_0^s(\Omega)$, an equivalent norm of $u$ is (see
\cite[Proposition 1.18]{MBRS})
\begin{align*}
	\|u\|_{X_0^s (\Omega)}^2 = \|u\|_{L^2(\Omega)}^2 + \left\| (-\Delta)^{\frac{s}{2}} u \right\|_{L^2(\mathbb{R}^N)}^2.
\end{align*}
More explicitly, for every $u \in X_0^s(\Omega)$,
\begin{align*}
\int_{\mathbb{R}^N \times \mathbb{R}^N} \frac{|u(x)-u(y)|^2}{|x-y|^{N+2s}}\, dx\, dy = \frac{2}{C(N,s)} \left\| (-\Delta)^{s/2} u \right\|^2_{L^2(\mathbb{R}^N)},
\end{align*}
where
\begin{align*}
C(N,s) &=\frac{s(1-s)}{A(N,s)B(s)}, \\
A(N,s) &= \int_{\mathbb{R}^{N-1}} \frac{d \eta}{(1+|\eta|^2)^{(N+2s)/2}}, \\
B(s) &= s (1-s) \int_{\mathbb{R}} \frac{1-\cos t}{|t|^{1+2s}} \, dt.
\end{align*}
\begin{Lem}
  For every $u \in H^1(\mathbb{R}^N)$, there results
  \begin{align*}
    \lim_{s \to 1^-} \left\| (-\Delta)^{s/2} u
    \right\|_{L^2(\mathbb{R}^N)}^2 = \left\| \nabla u
    \right\|_{L^2(\mathbb{R}^N)}^2.
  \end{align*}
\end{Lem}
\begin{proof}
  From \cite[Proposition 3.6]{DNPV}, we know that
  \begin{align*}
    \left\| (-\Delta)^{s/2} u \right\|_{L^2(\mathbb{R}^N)}^2 =
    \frac{C(N,s)}{2} \int_{\mathbb{R}^N \times \mathbb{R}^N}
    \frac{|u(x)-u(y)|^2}{|x-y|^{N+2s}}\, dx\, dy.
  \end{align*}
  From \cite[Remark 4.3]{DNPV}, we know that
  \begin{align*}
    \lim_{s \to 1^-} (1-s) \int_{\mathbb{R}^N \times \mathbb{R}^N}
    \frac{|u(x)-u(y)|^2}{|x-y|^{N+2s}}\, dx\, dy =
    \frac{\omega_{N-1}}{2N} \left\| \nabla u \right\|_{L^2(\mathbb{R}^N)}^2.
  \end{align*}
  Therefore, recalling \cite[Corollary 4.2]{DNPV},
  \begin{align*}
    \lim_{s \to 1^-} \left\| (-\Delta)^{s/2} u
    \right\|_{L^2(\mathbb{R}^N)}^2 &= \lim_{s \to 1^-}
                                     \frac{C(N,s)}{2(1-s)}
                                     \left( (1-s) \int_{\mathbb{R}^N \times \mathbb{R}^N}
    \frac{|u(x)-u(y)|^2}{|x-y|^{N+2s}}\, dx\, dy \right) \\
    &= \frac{1}{2} \frac{4N}{\omega_{N-1}}
      \frac{\omega_{N-1}}{2N} \left\| \nabla u
      \right\|_{L^2(\mathbb{R}^N)}^2 = \left\| \nabla u \right\|_{L^2(\mathbb{R}^N)}^2.
    \end{align*}
  \end{proof}
  
  On $X^s_0 (\Omega)$ we introduce a new norm
\begin{align}\label{norm-s}
\|u\|_s^2 := \left\| (-\Delta)^s u \right\|_{L^2(\mathbb{R}^N)}^2 + \int_\Omega V(x) u^2 \, dx, \quad u \in X^s_0 (\Omega),
\end{align}
which is, under (V), equivalent to $\| \cdot \|_{X_0^s
  (\Omega)}$. Similarly we introduce the norm on $H^1_0 (\Omega)$ by
putting
\begin{align}\label{norm-1}
\|u\|^2 := \int_{\Omega} |\nabla u|^2 + V(x) u^2 \, dx, \quad u \in H^1_0(\Omega).
\end{align}  
  
\begin{Cor} \label{cor:2.3}
  For every $u \in H_0^1(\Omega)$ we have
  \begin{align*}
    \lim_{s \to 1^-} \|u\|_s = \|u\|.
  \end{align*}
\end{Cor}
The following convergence result will be used in the sequel.
\begin{Lem}
  For every $ \varphi \in C_0^\infty(\Omega)$, there results
  \begin{align*}
    \lim_{s \to 1^-} \left\| (-\Delta)^{s} \varphi - (-\Delta) \varphi
    \right\|_{L^2(\Omega)} =0.
  \end{align*}
\end{Lem}
\begin{proof}
  We notice that
  \begin{multline*}
    \left\| (-\Delta)^{s} \varphi - (-\Delta) \varphi
    \right\|_{L^2(\Omega)} \leq \left\| (-\Delta)^{s} \varphi - (-\Delta) \varphi
    \right\|_{L^2(\mathbb{R}^N)} = \left\| \mathcal{F}^{-1}_\xi \left( \left(
    |\xi|^{2s} - |\xi|^2 \right) \hat{\varphi}(\xi) \right)
                             \right\|_{L^2(\mathbb{R}^N)} \\
    \leq C \left\| \left(
    |\cdot|^{2s} - |\cdot|^2 \right) \hat{\varphi}
                             \right\|_{L^2(\mathbb{R}^N)}
  \end{multline*}
  where $C>0$ is a constant, independent of $s$, that depends on the
  definition of the Fourier transform $\cF$. It is now easy to conclude,
  since the Fourier transform of a test function is a rapidly
  decreasing function.
\end{proof}

We will use the following embedding result.
\begin{Th}[\cite{MBRS}] \label{th:2.5}
  If $\Omega$ has a continuous boundary $\partial \Omega$, then the
  embedding $X_0^s(\Omega) \subset L^\nu(\Omega)$ is compact for every
  $1 \leq \nu < 2_s^*$, where $2_s^*=2N/(N-2s)$.
\end{Th} 
We will need some precise information on the embedding constant for
fractional Sobolev spaces. 
\begin{Th}[\cite{KT}]
  Let $N>2s$ and $2_s^*=2N/(N-2s)$. Then
  \begin{align} \label{eq:2.1}
    \|u\|_{L^{2_s^*} (\R^N)}^2 \leq \frac{\Gamma \left( \frac{N-2s}{2}
    \right)}{\Gamma \left( \frac{N+2s}{2} \right)} \left| \mathbb{S}
    \right|^{-\frac{2s}{N}} \| (-\Delta)^{s/2} u \|^2_{L^2 (\R^N)}
  \end{align}
  for every $u \in H^s(\mathbb{R}^N)$, where $\mathbb{S}$ denotes the $N$-dimensional unit sphere and $|\mathbb{S}|$ its surface area.
\end{Th}
\begin{Lem}\label{lemma-constant-independednt}
	Let $N \geq 3$ and $q \in [2,2N/(N-1)]$. Then there exists a constant $C=C(N,q)>0$ such that, for every $s \in [1/2,1]$ and every $u \in X_0^s(\Omega)$, we have
	\begin{align*}
	\|u\|_{L^{q} (\Omega)} \leq C(N,q) \| (-\Delta)^{s/2} u \|_{L^2 (\R^N)}.
	\end{align*}
\end{Lem}
\begin{proof}
  Since $\Gamma$ is a continuous function on the interval
  $\left[ \frac{N-2}{2},\frac{N}{2} \right]$ which does not contain
  non-positive integers, the constant
	\begin{align*}
	\frac{\Gamma \left( \frac{N-2s}{2}
		\right)}{\Gamma \left( \frac{N+2s}{2} \right)} \left| \mathbb{S}
	\right|^{-\frac{2s}{N}}
	\end{align*}
	in \eqref{eq:2.1} is bounded from above independently of
        $s \in [1/2,1]$. Therefore inequality \eqref{eq:2.1} holds true
        with a constant independent of $s$. Since obviously
	\begin{align*}
	\|u\|_{L^2 (\Omega)} \leq \|u\|_{X_0^s (\Omega)}
	\end{align*}
	for every $u \in X_0^s(\Omega)$, we can fix any $q \in
        [2,2N/(N-1)]$ and interpolate:
	\begin{align*}
	\frac{1}{q}=\frac{\vartheta_s}{2}+\frac{1-\vartheta_s}{2_s^*}.
	\end{align*}
	Explicitly,
	\begin{align*}
	\vartheta_s = \frac{2N-q(N-2s)}{2sq}
	\end{align*}
	and
	\begin{align*}
	\|u\|_{L^q (\Omega)} \leq \|u\|_{L^2 (\Omega)}^{\vartheta_s} \|u\|_{L^{2_s^*} (\Omega)}^{1-\vartheta_s}\leq C(N,q)^{1-\vartheta_s} \|u\|_{X_0^s(\Omega)}.
	\end{align*}
	Since the function $s \mapsto \vartheta_s$ is continuous in
        the interval $[1/2,1]$, the proof is complete.
\end{proof}
\begin{Rem}
  It follows from the previous proof that the same result is true for
  any $s \in [s_0,1]$, with $s_0 \in (0,1)$ fixed.
\end{Rem}
\begin{Def}
  A weak solution to problem \eqref{eq:1.1} is a function
  $u \in X_0^s(\Omega)$ such that
	\begin{align*}
	\langle (-\Delta)^{s/2} u \mid (-\Delta)^{s/2} \varphi \rangle_{L^2(\mathbb{R}^N)} + \int_\Omega V(x) u \varphi \, dx = \int_\Omega f(x,u) \varphi\, dx
	\end{align*}
	for every $\varphi \in X_0^s(\Omega)$.
\end{Def}
Weak solutions are therefore critical points of the associated energy
functional $\cJ_s \colon X^s_0 (\Omega) \rightarrow \R$ defined by
\begin{align*}
  \cJ_s (u) = \frac12 \left\| (-\Delta)^s u \right\|_{L^2(\mathbb{R}^N)}^2 + \frac12
  \int_{\Omega} V(x) u^2 \, dx - \int_{\Omega} F(x,u) \, dx.
\end{align*}

We recall also the definition of a weak solution in the local case.

\begin{Def}
A weak solution to problem \eqref{eq:1.2} is a function $u \in H^1_0 (\Omega)$ such that
\begin{align*}
\int_\Omega \nabla u \cdot \nabla \varphi \, dx + \int_\Omega V(x) u \varphi \, dx = \int_\Omega f(x,u) \varphi \, dx
\end{align*}
for every $\varphi \in H^1_0(\Omega)$.
\end{Def}

For the local problem \eqref{eq:1.2} we put
$\cJ\colon H^1_0 (\Omega) \rightarrow \R$
\begin{align} \label{eq:2.2}
  \cJ(u) = \frac12 \int_{\Omega} |\nabla u|^2 + V(x) u^2 \, dx - \int_{\Omega} F(x,u) \, dx.
\end{align}

Recalling the notation \eqref{norm-s} and \eqref{norm-1}, we can rewrite our functionals in the form
\begin{align*}
\cJ_s (u) &= \frac12 \|u\|_s^2 - \int_\Omega F(x,u) \, dx, \quad u \in X^s_0 (\Omega), \\
\cJ (u) &= \frac12 \|u\|^2 - \int_{\Omega} F(x,u) \, dx, \quad u \in H^1_0 (\Omega).
\end{align*}

\section{Existence of ground states}

We define the so-called Nehari manifolds
\begin{align*}
\cN_s := \{ u \in X^{s}_0 (\Omega) \setminus \{0\} \mid  \cJ_s' (u)(u) = 0 \}
\end{align*}
and
\begin{align*}
\cN := \{ u \in H^{1}_0 (\Omega) \setminus \{0\} \mid  \cJ' (u)(u) = 0 \}.
\end{align*}
\begin{Def}
A ground state of \eqref{eq:1.1} is any minimum point of $\cJ_s$ constrained on $\cN_s$. Similarly, a ground state of \eqref{eq:1.2} is any minimum point of $\cJ$ constrained on $\cN$. 
\end{Def}

To proceed, we show that ground states actually exist. 
\begin{Prop}\label{prop:existence}
  For every $s \in (0,1]$, there exists a ground state solution
  $u_s \in \mathcal{N}_s$ to \eqref{eq:1.1}. Moreover
  \begin{equation}\label{ground-state-level}
  \cJ_s (u_s) = \inf_{v \in X^s_0(\Omega) \setminus \{0\}} \sup_{t \in [0,1]} \cJ_s (t v) > 0.
  \end{equation}
\end{Prop}
\begin{proof}
The proof is rather standard, so we will present a sketch  and refer the reader to \cite{BM, SzW, SzW-Handbook} for the details.
  Consider $0<s<1$. It follows from our
  assumptions that the Nehari manifold $\cN_s$ is homeomorphic to the
  unit sphere $S_s$ in $X_0^s(\Omega)$. The homeomorphism $m_s \colon S_s \rightarrow \cN_s$ is given by
  \begin{align*}
  m_s(u) = t_u u,
  \end{align*}
where $t_u > 0$ is the unique positive number such that $t_u u \in \cN_s$. The inverse $m_s^{-1} \colon \cN_s \rightarrow S_s$ is given by $m_s^{-1}(u) = u / \|u\|_s$. Moreover $\cJ_s \circ m_s \colon S_s \rightarrow \R$ is still of class $C^1$. Then there is a Palais-Smale sequence $\{v_n\}_n \subset S_s$   for $\cJ \circ m_s$. Moreover, we can show that the sequence $\{u_n\}_n \subset \cN_s$ given by $u_n := m_s (v_n)$ is a bounded Palais-Smale sequence for $\cJ_s$ such that $\cJ(u_n) \to c_s$, where $c_s := \inf_{\cN_s} \cJ_s > 0$. Since $X_0^s(\Omega)$ is compactly embedded into $L^\nu(\Omega)$ for every $2 \leq \nu < 2_s^*$, see Theorem~\ref{th:2.5}, it is easy to check that $\{u_n\}_n$ converges strongly (up to a subsequence) in $L^\nu(\Omega)$ to a function $u \neq 0$ such that $\cJ_s'(u)=0$. Finally, the properties of $F$ yield
  \begin{align*}
    \cJ_s(u) = \frac{1}{2} \|u\|_s^2 - \int_\Omega F(x,u)\, dx \leq \liminf_{n \to +\infty} \left\{ \frac{1}{2} \|u_n\|_s^2 -
             \int_\Omega F(x,u_n)\, dx \right\}  
    = \liminf_{n \to +\infty}  \cJ_s(u_n) = c_s.
  \end{align*}
  The proof for the case $s=1$ is similar.
\end{proof}

\section{Non-local to local transition}

For any $s \in (1/2,1)$ we define
\begin{align*}
c_s := \inf_{\cN_s} \cJ_s > 0.
\end{align*}
Similarly, we put also
\begin{align*}
c := \inf_{\cN} \cJ > 0.
\end{align*}

For any $v \in X^s_0 (\Omega) \setminus \{0\}$ we let
$t_s(v) > 0$ be the unique positive real number such that $t_s (v) \in \cN_s$. Then we put $m_s(v) := t_s (v) v$ (see the proof of Proposition \ref{prop:existence}).

\begin{Lem}\label{lem:limsup}
	There results
\begin{align*}
\limsup_{s \to 1^-} c_s \leq c.
\end{align*}
\end{Lem}
\begin{proof}
  Take $u \in H^1_0 (\Omega) \subset X^s_0 (\Omega)$ as a ground state
  solution of \eqref{eq:1.2}, in particular $u \in \cN$ and
  $\cJ(u) = c$, where $\cJ$ is given by \eqref{eq:2.2}. Consider the function $m_s (u) \in \cN_s$. Obviously
\begin{align*}
c_s \leq \cJ_s (m_s (u)).
\end{align*}
Hence
\begin{multline*}
  \limsup_{s \to 1^-} c_s \leq \limsup_{s \to 1^-} \cJ_s (m_s(u)) =
  \limsup_{s \to 1^-}  \left\{ \cJ_s (m_s(u))  - \frac12 \cJ_s '
    (m_s(u)) \right\} \\
  = \limsup_{s \to 1^-} \left\{ \frac12 \int_\Omega f(x,m_s(u))m_s(u) - 2
  F(x,m_s(u)) \, dx \right\}.
\end{multline*}
Recall that $m_s(u) = t_s u$ for some real numbers $t_s > 0$. Suppose
by contradiction that $t_s \to +\infty$ as $s \to 1^-$. Then, in view
of the Nehari identity
\begin{align*}
\| u \|_{s}^2 = \int_\Omega \frac{f(x,t_s u)}{t_s^2} t_s u \, dx
\geq 2 \int_\Omega \frac{F(x,t_s u)}{t_s^2 u^2} u^2 \, dx \to +\infty,
\end{align*}
but the left-hand side stays bounded (see Corollary \ref{cor:2.3}). Hence $(t_s)_s$ is bounded. Take any
convergent subsequence $(t_{s_n})$ of $(t_s)$, i.e. $t_{s_n} \to t_0$ as $n\to+\infty$. Obviously $t_0 \geq 0$. We will show that $t_0 \neq 0$. Indeed, suppose that $t_0 = 0$, i.e. $t_{s_n} \to 0$. Then, in view of the Nehari identity
\begin{align*}
\| u \|_{s_n}^2 = \int_\Omega \frac{f(x,t_{s_n} u)}{t_{s_n} u} u^2 \, dx.
\end{align*}
By Corollary \ref{cor:2.3}, $\| u \|_{s_n}^2 \to \|u\|^2 > 0$. Hence, in view of (F2),
\begin{align*}
\|u\|^2 + o(1) = \int_\Omega \frac{f(x,t_{s_n} u)}{t_{s_n} u} u^2 \, dx \to 0,
\end{align*}
a contradiction. Hence $t_0 > 0$. Again, by Corollary \ref{cor:2.3},
\begin{align*}
t_{s_n}^2 \| u \|_{s_n}^2 \to t_0^2 \|u \|^2 \quad \mathrm{as} \ n\to +\infty.
\end{align*}
Moreover, in view of Remark \ref{rem:1.1},
\begin{align*}
 |f(x,t_{s_n} u) t_{s_n} u| \leq  \varepsilon t_{s_n}^2 |u|^2 + C_\varepsilon t_{s_n}^p |u|^p \leq C ( |u|^2 + |u|^p)
\end{align*}
for some constant $C > 0$, independent of $n$. In view of the Lebesgue's convergence theorem
\begin{align*}
\int_\Omega f(x,t_{s_n} u) t_{s_n} u \, dx \to \int_\Omega f(x,t_0 u) t_0 u \, dx.
\end{align*}
Thus the limit $t_0$ satisfies
\begin{align*}
t_0^2 \|u\|^2 = \int_\Omega f(x,t_0 u) t_0 u \, dx.
\end{align*}
Taking the Nehari identity into account we see that $t_0 = 1$. Hence
$t_s \to 1$ as $s \to 1^-$. Repeating the same argument we see that
\begin{multline*}
\limsup_{s \to 1^-} \left\{ \frac12 \int_\Omega f(x,m_s(u))m_s(u) - 2
  F(x,m_s(u)) \, dx \right\} = \frac12 \int_\Omega f(x,u)u - 2 F(x,u)
                               \, dx \\
  = \cJ(u) = c
\end{multline*}
and the proof is completed.
\end{proof}

\begin{Lem}\label{lemma-boundedness}
There exists a constant~$M > 0$ such that
\begin{align*}
\|u_s\|_{L^2 (\Omega)} + \|u_s\|_{s} \leq M
\end{align*}
for every $s \in (1/2,1)$.
\end{Lem}
\begin{proof}
  Note that $\|u_s\|_{L^2 (\Omega)} \leq C \|u_s\|_{s}$, for some $C > 0$ independent of $s$. So it is enough to
  show that $\|u_s\|_{s} \leq M$. Suppose by contradiction that
\begin{align*}
\|u_s\|_{s} \to +\infty \quad \mathrm{as} \ s \to 1^-.
\end{align*}
Put $v_s := \frac{u_s}{\|u_s\|_s}$. Then $\|v_s\|_{s} = 1$, so \cite[Corollary 7]{BBM} implies that $v_s \to v_0$ in $L^2(\Omega)$ for some $v_0 \in H^1_0(\Omega)$. From Lemma \ref{lemma-constant-independednt} $\{v_s\}_s$ is bounded in $L^{\frac{2N}{N-1}} (\Omega)$. Take any $\nu \in \left(2, \frac{2N}{N-1} \right)$ and by the interpolation inequality
\begin{align*}
\|v_s - v_0 \|_{L^\nu (\Omega)} \leq \| v_s - v_0 \|_{L^2 (\Omega)}^\vartheta \|v_s - v_0 \|_{L^{\frac{2N}{N-1}} (\Omega)}^{1-\vartheta} \to 0 \quad \mathrm{as} \ s \to 1^-,
\end{align*}
where $\vartheta \in (0,1)$ is chosen so that $\frac{1}{\nu} = \frac{\vartheta}{2} + \frac{1-\vartheta}{\frac{2N}{N-1}}$. Hence $v_s \to v_0$ in $L^\nu (\Omega)$ for all $2 \leq \nu < \frac{2N}{N-1}$. In particular, we can choose a sequence $\{v_{s_n}\}_n$ such that $v_{s_n} (x) \to v_0(x)$ for a.e. $x \in \Omega$. Note that, from Lemma \ref{lem:limsup}, we know that $\{\cJ_{s_n} (u_{s_n})\}_n$ is bounded. We will consider two cases.

\begin{itemize}
\item Suppose that $v_0 = 0$. Fix any $t > 0$. By \eqref{ground-state-level} we obtain
\begin{align*}
\cJ_{s_n} (u_{s_n}) \geq \cJ_{s_n} \left( \frac{t}{\|u_{s_n}\|_{s_n}} u_{s_n} \right) = \cJ_{s_n} (t v_{s_n}) = \frac{t^2}{2} - \int_\Omega F(x,tv_{s_n}) \, dx.
\end{align*}
From Remark \ref{rem:1.1} we see that
\begin{align*}
\int_\Omega F(x,tv_{s_n}) \, dx \leq \varepsilon t^2 \| v_{s_n}\|_{L^2(\Omega)}^2 + C_\varepsilon t^p \|v_{s_n}\|_{L^p (\Omega)}^p \to 0.
\end{align*}
Hence, for any $t > 0$
\begin{align*}
\cJ_{s_n} (u_{s_n}) \geq \frac{t^2}{2} + o(1),
\end{align*}
which is a contradiction with the boundedness of $\{\cJ_{s_n}(u_{s_n})\}_n$. 
\item Suppose that $v_0 \neq 0$, i.e. $| \supp v_0 | > 0$. Note that for a.e. $x \in \supp v_0$ we have
\begin{align*}
|u_{s_n}(x)| = \| u_{s_n} \|_{s_n} | v_{s_n} (x)| \to +\infty.
\end{align*}
Hence, taking into account the boundedness of $\{\cJ_{s_n} (u_{s_n}) \}_n$ and  Fatou's lemma,
\begin{align*}
o(1) = \frac{\cJ_{s_n} (u_{s_n})}{\|u_{s_n}\|_{s_n}^2} = \frac12 - \int_{\Omega} \frac{F(x, u_{s_n})}{u_{s_n}^2} v_{s_n}^2 \, dx \leq \frac12 - \int_{\supp v_0} \frac{F(x, u_{s_n})}{u_{s_n}^2} v_{s_n}^2 \, dx \to -\infty,
\end{align*}
again a contradiction.
\end{itemize}
\end{proof}

\begin{Cor} \label{cor:2.9}
There is $u_0 \in H^1_0 (\Omega)$ and a sequence $\{s_n\}_n$ such that $s_n \to 1^-$ and
\begin{align*}
u_{s_n} \to u_0 \quad \mathrm{in} \ L^\nu (\Omega) \quad \mathrm{as} \ n \to +\infty
\end{align*}
for all~$\nu \in [2, 2N/(N-1))$.
\end{Cor}

\begin{proof}
From Lemma \ref{lemma-boundedness} and \cite[Corollary 7]{BBM} we note that
\begin{align*}
u_{s_n} \to u_0 \quad \mathrm{in} \ L^2(\Omega)
\end{align*}
for some $u_0 \in H^1_0 (\Omega)$ and sequence $\{s_n\}_n$. In view of Lemma
\ref{lemma-constant-independednt} there is a constant~$C > 0$ (independent of
$s$) such that
\begin{align*}
\|u_{s_n}\|_{ L^{\frac{2N}{N-1}}(\Omega) } \leq C \| u_{s_n} \|_{s_n}.
\end{align*}
In particular, $\{u_{s_n}\}_n$ is bounded in $L^{\frac{2N}{N-1}}
(\Omega)$. Then for any $\nu \in (2, 2N/(N-1))$ we have
\begin{align*}
 \limsup_{n \to +\infty}  \|u_{s_n} - u_0\|_{L^\nu (\Omega)} \leq \limsup_{n \to +\infty} \|u_{s_n} - u_0 \|_{L^2 (\Omega)}^\vartheta \| u_{s_n} - u_0 \|_{L^{\frac{2N}{N-1}} (\Omega)}^{1-\vartheta} =0,
\end{align*}
where $\vartheta \in (0,1)$ is chosen so that 
\begin{align*}
\frac{1}{\nu} = \frac{\vartheta}{2} + \frac{1-\vartheta}{\frac{2N}{N-1}}.
\end{align*}
\end{proof}

\begin{Lem}
The limit $u_0$ is a weak solution for \eqref{eq:1.2}.
\end{Lem}

\begin{proof}
Take any test function $\varphi \in C_0^\infty (\R^N)$ and note that by \cite[Section 6]{W}
we have
\begin{align*}
\int_{\R^N} (-\Delta)^{s/2} u_{s_n} (-\Delta)^{s_n /2} \varphi \, dx = \int_{\R^N} u_{s_n} (-\Delta)^{s_n} \varphi \, dx.
\end{align*}
Moreover
\begin{multline*}
\left| \int_{\R^N} u_{s_n} (-\Delta)^{s_n} \varphi \, dx - \int_{\R^N} u_0 (-\Delta \varphi) \, dx \right| = \left| \int_{\Omega} u_{s_n} (-\Delta)^{s_n} \varphi \, dx - \int_{\Omega} u_0 (-\Delta \varphi) \, dx \right| \\
= \left| \int_{\Omega} u_{s_n} \left( (-\Delta)^{s_n} \varphi - (-\Delta \varphi) \right) \, dx + \int_{\Omega} (u_{s_n} - u_0) (-\Delta \varphi) \, dx \right| \\
\leq \left\|u_{s_n} \right\|_{L^2(\Omega)} \left\| (-\Delta)^{s_n} \varphi - (-\Delta \varphi) \right\|_{L^2(\Omega)} + \|  (-\Delta \varphi) \|_{L^2(\Omega)} \|u_{s_n} - u_0\|_{L^2(\Omega)} \to 0.
\end{multline*}
Hence
\begin{align*}
\lim_{n \to +\infty} \int_{\R^N} (-\Delta)^{s_n /2} u_{s_n} (-\Delta)^{s_n /2} \varphi \, dx = \int_{\R^N} u_0 (-\Delta \varphi) \, dx = \int_{\Omega} \nabla u_0 \cdot \nabla \varphi \, dx.
\end{align*}
Obviously
\begin{align*}
\lim_{n \to +\infty} \int_\Omega V(x) u_{s_n} \varphi \, dx = \int_\Omega V(x) u_0 \varphi \, dx.
\end{align*}
Take any measurable set $E \subset \Omega$ and note that, taking into account Remark \ref{rem:1.1},
\begin{align*}
\int_E | f(x,u_{s_n}) \varphi | \, dx \leq \varepsilon \| u_{s_n} \|_{L^2 (\Omega)} \| \varphi \chi_E \|_{L^2 (\Omega)} + C_\varepsilon \| u_{s_n} \|_{L^p (\Omega)}^{p-1} \| \varphi \chi_E \|_{L^p (\R^N)}.
\end{align*}
Hence the family $\{ f(\cdot, u_{s_n}) \varphi \}_n$ is uniformly integrable on $\Omega$ and in view of the Vitali convergence theorem
\begin{align*}
\lim_{n \to +\infty} \int_\Omega f(x,u_{s_n})\varphi \, dx = \int_\Omega f(x,u_0)\varphi \, dx.
\end{align*}
Therefore $u_0$ satisfies
\begin{align*}
\int_{\Omega} \nabla u_0 \cdot \nabla \varphi \, dx + \int_\Omega V(x) u_0 \varphi \, dx = \int_\Omega f(x,u_0)\varphi \, dx,
\end{align*}
i.e. $u_0$ is a weak solution to \eqref{eq:1.2}.
\end{proof}
\begin{Lem}
Since $u_s \in \cN_s$ there is (independent of $s$) constant $\rho$ such that
\begin{align*}
\|u_s\|_{s} \geq \rho > 0.
\end{align*}
\end{Lem}
\begin{proof}
  Since $u_s \in \mathcal{N}_s$, we can write by Remark \ref{rem:1.1}
  \begin{align*}
    \|u_s\|_{s}^2 = \int_{\Omega} f(x,u_s)u_s \, dx &\leq \varepsilon \|u_s\|_{L^2(\Omega)}^2 + C_\varepsilon \|u_s\|_{L^p(\Omega)}^p \\
                                                    &\leq C \left( \varepsilon \|u_s\|_{s}^2 + C_\varepsilon \|u_s\|_{s}^p \right)
  \end{align*}
  for a constant $C>0$ independent of $s$. Choosing $\varepsilon>0$
  small enough, we conclude that
  \begin{align*}
    \|u_s\|_{s}^{p-2} \geq \frac{1-C \varepsilon}{C \cdot C_\varepsilon} =: \rho>0.
	\end{align*}
\end{proof}

\begin{Lem}
We have~$u_0 \neq 0$ and therefore $u_0 \in \cN$.
\end{Lem}

\begin{proof}
If $u_0 = 0$, then $u_{s_n} \to 0$ in $L^2(\Omega)$ and in $L^p(\Omega)$. Then
\begin{align*}
\|u_{s_n}\|_{s_n} = \int_\Omega f(x,u_{s_n}) u_{s_n} \, dx \to 0,
\end{align*}
a contradiction.
\end{proof}

\begin{Lem}\label{lem:liminf}
  There results
\begin{align*}
\liminf_{n \to +\infty} c_{s_n} \geq c.
\end{align*}
\end{Lem}

\begin{proof}
Since $u_{s_n} \in \mathcal{N}_{s_n}$, then by Corollary \ref{cor:2.9} we have
\begin{align*}
\liminf_{n \to +\infty} c_{s_n} &= \liminf_{n \to +\infty} \cJ_{s_n} (u_{s_n}) = \liminf_{n \to +\infty} \left\{ \cJ_{s_n} (u_{s_n}) - \frac12 \cJ_{s_n} ' (u_{s_n})(u_{s_n}) \right\} \\
&= \liminf_{n \to +\infty} \left\{ \frac{1}{2} \int_\Omega f(x,u_{s_n})u_{s_n} -
  2 F(x,u_{s_n}) \, dx  \right\} \\
&=  \frac12 \int_\Omega f(x,u_0)u_0 - 2 F(x,u_0) \, dx = \cJ(u_0) \geq c.
\end{align*}
\end{proof}
\begin{Lem}\label{lem:ground-state}
The function $u_0 \in H^1_0(\Omega)$ is a ground state solution to \eqref{eq:1.2}.
\end{Lem}
\begin{proof}
Note that, from Lemma \ref{lem:limsup} and \ref{lem:liminf} we have
\begin{align*}
\liminf_{n \to + \infty} c_{s_n} \geq c \geq \limsup_{s \to 1^-} c_s \geq \limsup_{n \to + \infty} c_{s_n}.
\end{align*}
Hence $\lim_{n \to +\infty} c_{s_n}$ exists and $\lim_{n \to +\infty} c_{s_n} = c$. From the proof of Lemma \ref{lem:liminf} we have
\begin{align*}
\lim_{n \to +\infty} c_{s_n} = \cJ(u_0).
\end{align*}
Thus $\cJ(u_0) = c$.
\end{proof}

\begin{proof}[Proof of Theorem \ref{th:main}]
The statement is a direct consequence of Corollary \ref{cor:2.9} and Lemma \ref{lem:ground-state}.
\end{proof}

\section*{Acknowledgements}

Bartosz Bieganowski was partially supported by the National Science Centre, Poland (Grant No. 2017/25/N/ST1/00531). Simone Secchi is member of the \emph{Gruppo Nazionale per l'Analisi Ma\-te\-ma\-ti\-ca, la Probabilit\`a e le loro Applicazioni} (GNAMPA) of the {\em Istituto Nazionale di Alta Matematica} (INdAM).

\end{document}